\documentclass[11pt, reqno]{amsart}    

\usepackage{amsmath,amssymb,latexsym}
\usepackage{color}
\usepackage{xcolor}
\usepackage{graphicx}

\newcommand{\R}{\mathbb{R}}
\newcommand{\T}{\mathbb{T}}
\newcommand{\Z}{\mathbb{Z}}
\newcommand{\N}{\mathbb{N}}
\newcommand{\K}{\mathcal{K}}

\newcommand{\diff}{\mathrm{d}}

\newtheorem{theorem}{Theorem}[section]
\newtheorem{lemma}[theorem]{Lemma}
\newtheorem{proposition}[theorem]{Proposition}

\newtheorem*{main-theorem}{Main Theorem}
\newtheorem*{remark*}{Remark}
\newtheorem*{lemma*}{Lemma A.1}

\numberwithin{equation}{section}

\begin{document}

\title[Conditional well-posedness in the bidirectional Whitham]{A conditional well-posedness result for the bidirectional Whitham equation}

\author{Mats Ehrnstr\"om}
\address{Department of Mathematical Sciences, NTNU Norwegian University of Science and Technology, 7491 Trondheim, Norway.}
\email{mats.ehrnstrom@ntnu.no}

\author{Long Pei}
\address{Department of Mathematics, KTH Royal Institute of Technology, 10044  Stockholm, Sweden}
\email{longp@kth.se}

\author{Yuexun Wang}
\address{Department of Mathematical Sciences, Norwegian University of Science and Technology, 7491 Trondheim, Norway.}
\email{yuexun.wang@ntnu.no}

\thanks{M.E. and Y.W. acknowledge the support by grants nos. 231668 and 250070 from the Research Council of Norway.}


\begin{abstract}
We consider the initial-value problem for the bidirectional Whitham equation, a system which combines the full two-way dispersion relation from the incompressible Euler equations with a canonical shallow-water nonlinearity. We prove local well-posedness in classical Sobolev spaces in the localised as well as the periodic case, using a square-root type transformation to symmetrise the system. The existence theory requires a non-vanishing surface elevation, indicating that the problem is ill-posed for more general initial data.
\end{abstract}

\maketitle

\section{Introduction and main results}

We consider the  bidirectional Whitham equation
\begin{equation}\label{eq:bdw}
\begin{aligned}
\partial_t\eta &=-\mathcal{K}\partial_xu-\partial_x(\eta u)\\
\partial_tu &=-\partial_x\eta-u\partial_xu,
\end{aligned}
\end{equation}
formally derived in \cite{aceves2013numerical,moldabayev2015whitham} from the incompressible Euler equations to model fully dispersive shallow water waves whose propagation is allowed to be both left- and rightward. Here,  $\eta$ denotes the surface elevation, \(u\) is the rightward velocity at the surface, and the Fourier multiplier operator $\mathcal{K}$ is  defined by
\begin{equation}\label{eq:K}
\widehat {{\mathcal{K}v}}(\xi)=\frac{\tanh(\xi)}{\xi} \, \widehat{v}(\xi),
\end{equation}
for all \(v\) in the Schwartz space \(\mathcal{S}(\mathbb{R})\). By duality, the operator \(\K\) is well-defined on the space of tempered distributions, \(\mathcal{S}'(\R)\).  The model \eqref{eq:bdw} is the two-way equivalent of the Whitham equation 
\begin{equation}\label{eq:whitham}
u_t + \K^{\frac{1}{2}} u_x + u u_x = 0, 
\end{equation}
a nonlocal shallow water equation that in its simple form still captures several interesting mathematical features that are present also in the full water-wave problem. The operator \(\K^{\frac{1}{2}}\) is the square root of the operator \(\K\) defined in \eqref{eq:K}, most easily defined by considering the action of these operators in Fourier space. The features of \eqref{eq:whitham} include solitary \cite{EGW12} and heighest \cite{EWhighest} waves, finite-time breaking \cite{H15} and modulational instability \cite{HJ15}. 

Two-way fully dispersive systems related to the uni- and bidirectional Whitham equations have been considered for example in \cite{hur2016wave} and \cite{saut2012well}. The kind of full system one would ultimately like to handle is something akin to (1.1) in \cite{saut2012well} (the same system appears in \cite{lannes2013water}) The other, simpler, systems may be viewed both as steps in this direction and as models in their own right. In case of the bidirectional Whitham equation \eqref{eq:bdw}, it is mathematically interesting because of its weak dispersion, and contains a logarithmically cusped wave of greatest height \cite{ehrnstrom2016existence}. Experiments indicate surprisingly good modelling properties for this model, as well as for several other 'Whitham-like' equations and systems; they significantly outperform the KdV equation in the experimental setting, see \cite{Carter2017} and \cite{MR3523508}. Still, we regard our result as a mathematical one: the system \eqref{eq:bdw} is well-posed, but the set of initial-data for which we can control the life-span is bounded away from a zero surface deflection.\footnote{It is an interesting question how this aligns with the experimental data in \cite{Carter2017}, apparently not displaying this shortcoming. One possibility is that classical Sobolev spaces are too large for the system \eqref{eq:bdw}. We do not provide an answer, but want to make the reader aware of these facts.}
A similar observation, without proof, has been made very recently in \cite{klein2017whitham}, an investigation written in parallel to our paper and that establishes the validity of the Whitham equation as a water-wave model, in the KdV regime. Although we use standard energy methods, we hope the reader will be convinced that there are some details to be made. 

The weak dispersion of \eqref{eq:bdw} clearly suggests to view it as a perturbation of a hyperbolic system. One could symmetrise the system in many ways, for example by using matrices with diagonals \((1,\eta)\) or \((1/\eta, 1)\). In this paper, we adopt the transformation \(\eta\mapsto\sqrt{\eta}\), sometimes used in  physical settings as a sound speed transformation and in the blow-up analysis in fluid mechanics (cf. \cite{dafermos2010hyperbolic}), to transfer the system \eqref{eq:bdw} into a canonical  form. Although this does not provide any better results for smooth data, such a transformation may be of benefit when regarding weaker data. The differences in the analysis between the three different symmetrisations are minimal, and in all three cases  it should be emphasised that one needs a positive lower bound for $\eta$ to ensure the symmetriser is positive definite or to remove the singularity in the canonical form when using the transformation with the square root. For the same reason the standard argument in \cite{kato1975cauchy, majda2012compressible} does not directly apply because the matrices in the hyperbolic term are only in homogeneous Sobolev spaces in the presence of a positive background.  Finally, we view our result within a broader framework, a program to investigate the interplay between dispersive and nonlinear effects in nonlocal equations, and aim to continue to investigate what solutions and properties similar equations allow for.

To state the results, let $\mathbb{X}$ be either the real line \(\R\) or the torus \(\T\) of circumference \(2\pi\), and let \(N \geq 2\) be an integer. Our main result is then as follows.

\begin{theorem}\label{thm:main} 
Let $(\eta_0,u_0)$ be initial data such that \(\inf \eta_0 > 0\) and
\begin{equation}\label{3}
(\sqrt{\eta_0}-\sqrt{\bar{\eta}},u_0){\color{purple}}\in H^N(\mathbb{X})
\end{equation}
for some positive constant \(\bar{\eta}\). Then the equation \eqref{eq:bdw} is locally well-posed. There exist a positive time $T > 0$ and a classical solution $(\eta,u)^{tr}$ of \eqref{eq:bdw} with  \((\eta,u)|_{t=0} = (\eta_0, u_0)\) that is unique among solutions satisfying
\begin{equation*}
  (\sqrt{\eta}-\sqrt{\bar{\eta}},u)\in C([0,T];H^N(\mathbb{X}))\cap C^1([0,T];H^{N-1}(\mathbb{X})).
\end{equation*}
The solution depends continuously on  $(\eta_0,u_0)$ with respect to the same metric.
\end{theorem}

It should be noted that in the statement of Theorem~\ref{thm:main} the constant \(\bar{\eta}\) is fixed, whence the metric is fixed, too.  The proof of Theorem \ref{thm:main} is presented throughout Sections~\ref{sec:prel}--\ref{proof of main theorem}, and we comment on the periodic case in a separate appendix. Section~\ref{sec:prel} contains the statement and reformulation of the problem, as well as necessary preliminaries. In Section~\ref{sec:regularised} we obtain a short-time existence result for the linearised and regularised problem.  The Gagliardo--Nirenberg interpolation inequality and tame product estimates (see \cite{hamilton1982inverse})  taylored to the non-vanishing background are used to show that the space-dependent part of the problem defines a bounded map on $H^N(\mathbb{R})$, thereby reducing our problem to an ODE in this space. Section 4 includes the main estimates of this letter, used to obtain convergence from the regularised linear problem to the purely linear one. The good properties of mollifiers $\mathcal{J}_\varepsilon$ are then deployed to get uniform estimates, and to treat the inhomogeneous terms in the equation. This plays an important role in finding a Cauchy sequence in $C({[0,T_2]; L^2 (\mathbb{R})})$, converging to a solution of the problem (in a better regularity class).  We mention that the same techniques may be applied to any of the symmetrisations mentioned in this paper to overcome the obstacle of the positive background introduced via \(\eta\).

\section{Preliminaries and setup of the problem}\label{sec:prel}
With \(\mathbb{X} \in \{\R,\T\}\) as above, let \(L^p(\mathbb{X})\), \(p \in [1,\infty]\), be the standard Lebesgue spaces with inner product 
\[
(f,g)_2=\int_{\mathbb{X}}fg \, \diff x
\]
in the case \(p=2\). Similarly, let $H^s(\mathbb{X}) = (1-\partial_x^2)^{-s/2} L^2(\mathbb{X})$ be the Bessel-potential spaces with norm 
\[
\|\cdot\|_{H^s(\mathbb{X})} = \|(1-\partial_x^2)^{s/2} \cdot \|_{L^2(\mathbb X)}, \qquad s \in \R,
\]
and for any Banach space \(\mathbb Y\), let $C^k([0,T];\mathbb{Y})$ be the space of all bounded continuous functions $u\colon [0,T]\rightarrow \mathbb{Y}$ with bounded and continuous derivatives up to \(k\)th order, normed by
\[
\| f \|_{C([0,T];{\mathbb Y})} =  \sum_{j=0}^k \sup_{t \in [0,T]} \|\partial_t^j f(t,\cdot)\|_{\mathbb Y}. 
\]
We write $f\lesssim g$ when $f\leq   cg$ for some constant \(c > 0\), and \(f \eqsim g\) when \(f \lesssim g \lesssim f\).  Finally, for a given positive constant \(\bar\eta\) and any function \(\eta\), let
\[
 \bar\lambda=\lambda(\bar{\eta}) \quad\text{ and }\quad \zeta=2(\lambda(\eta)-\bar\lambda),
\]
where \(\lambda=\sqrt{\cdot}\) is a shorthand to ease notation.
Then \eqref{eq:bdw} may be expressed as 
\begin{align*}
\partial_t\zeta+u\partial_x\zeta+\frac{\zeta+2\bar\lambda}{2}\partial_xu+\frac{2}{\zeta+2\bar\lambda}\mathcal{K}\partial_xu &=0,\\
\partial_tu+u\partial_xu+\frac{\zeta+2\bar\lambda}{2}\partial_x\zeta &=0,
\end{align*}
or, with
\[
U = 
\begin{pmatrix}
\zeta \\
u
\end{pmatrix}, \quad A(U)=
\begin{pmatrix}
u & \frac{\zeta+2\bar\lambda}{2} \\
\frac{\zeta+2\bar\lambda}{2} & u
\end{pmatrix} \quad\text{and} \quad B(U)=
\begin{pmatrix}
0 & \frac{2}{\zeta+2\bar\lambda} \\
0 & 0
\end{pmatrix},
\]
as
\begin{eqnarray}\label{5}
\partial_tU+A(U)\partial_xU+B(U)\mathcal{K}\partial_xU=0.
\end{eqnarray}
The system \eqref{5} is hyperbolic with a nonlocal dispersive perturbation and we shall look for solutions in Sobolev spaces embedded into \(L^\infty(\R)\). One notes that the initial data \(\zeta_0=2(\lambda(\eta_0)-\bar\lambda)\) satisfies \(\zeta_0 + 2\bar\lambda \geq 2 \sqrt{\inf \eta_0} > 0\) and may thus pick a positive constant \(\mu\) such that \(\bar \lambda \leq \mu^{-1}\) and
\begin{equation}\label{eq:assume u}
2\mu \leq \zeta_0 + 2\bar\lambda \leq (2\mu)^{-1},
\end{equation}
that we will use below.  The initial data \(U(0,x)\) for our problem shall be denoted by
\begin{equation}\label{6}
U_0 = (\zeta_0,u_0)^{tr},
\end{equation}
where \emph{tr} denotes the transpose of a matrix. Finally, let $U^{(k)}=(\partial_x^k\zeta,\partial_x^ku)^{tr}$, and define the partial and total energy functionals as
\begin{eqnarray*}
	E^{(k)}(t,U) =\|U^{(k)}(t,\cdot)\|_{L^2(\mathbb{X})}^2=\|\zeta^{(k)}(t,\cdot)\|_{L^2(\mathbb{X})}^2
	+\|u^{(k)}(t,\cdot)\|_{L^2(\mathbb{X})}^2
\end{eqnarray*}
and
\begin{eqnarray*}
	E_N(t,U) =\sum_{k=0}^NE^{(k)}(t,U),
\end{eqnarray*}
respectively. We will always assume that the integer $N\geq2$.  We shall sometimes write simply $E_N(t)$,  and similarly $E_N(U_0)$ will mean $E_N(0,U)$.

\section{The regularised and linearised problem}\label{sec:regularised}
For $0 < \varepsilon \ll 1$, let $\mathcal{J}_\varepsilon$ be a standard mollifier based on some smooth and compactly supported function $\varrho$ on $\R$. Denote by $\N_{0}$ the set of non-negative integers. We consider first the regularised problem
\begin{equation}\label{eq:regularised}
\partial_tU_\varepsilon+\mathcal{J}_\varepsilon [\mathcal{J}_\varepsilon(A(V))\partial_x(\mathcal{J}_\varepsilon U_\varepsilon)]+\mathcal{J}_\varepsilon[\mathcal{J}_\varepsilon(B(V))\mathcal{K}\partial_x(\mathcal{J}_\varepsilon U_\varepsilon)]=0,
\end{equation}
with initial data \(U_\varepsilon(0,x)=U_0(x)\).  Here, for any positive number $T_1$, it is assumed that 
\[
V=(\varphi,v)^{tr}\in C([0,T_1];H^N(\mathbb{R}))\cap C^1([0,T_1];H^{N-1}(\mathbb{R}))
\] 
satisfies
\begin{equation}\label{eq:assume v}
\begin{aligned}
  E_N(t,V) &\leq 2E_N(U_0),\\ 
  \mu \leq\varphi+2\bar\lambda &\leq \mu^{-1},
\end{aligned}  
\end{equation}
for all \((t,x) \in  [0,T_1]\times \mathbb{R}\). We will make repeated use of the following estimates.

\begin{lemma}\label{lemma:mollification} Mollification is continuous \(L^\infty \to BUC\), and for $k, l \in \N_0$,
\begin{align*}
\|\mathcal{J}_\varepsilon f\|_{H^{k+l}(\mathbb{R})} &\lesssim \varepsilon^{-l}\|f\|_{H^k(\mathbb{R})},\\
\|(\mathcal{J}_\varepsilon-\mathcal{J}_{\varepsilon'})f\|_{H^{k}(\mathbb{R})} &\lesssim |\varepsilon-\varepsilon'| \|\partial_xf\|_{H^{k}(\mathbb{R})}.
\end{align*}
\end{lemma}
\begin{proof}
The \(L^\infty \to BUC\)-continuity is a consequence of Young's inequality \cite{tartar2007introduction}, whereas the \(H^{k+l} \to H^k\)-estimate can be found for example in \cite{bona1975initial,kato1987nonstationary}. The last estimate is a small twist (to annihilate the constant in \(\varphi + 2 \bar\lambda\)) on the standard estimate \(\|\mathcal{J}_\varepsilon f-f\|_{H^{k}(\mathbb{R})}\lesssim \varepsilon \|f\|_{H^{k+1}(\mathbb{R})}\) , which may be proved in the following way: there exists \(s \in (0,1)\) such that
\[
(\mathcal{J}_\varepsilon-\mathcal{J}_{\varepsilon'})f(x) = (\varepsilon'-\varepsilon)\int_{\mathbb{R}}z\varrho(z)\partial_xf(x-s\varepsilon z-(1-s)\varepsilon' z)\, \diff z.
\]
Therefore, 
\begin{align*}
	&\|\partial_x^l(\mathcal{J}_\varepsilon-\mathcal{J}_{\varepsilon'})f\|_{L^2(\mathbb{R})}\\
	& \leq|\varepsilon-\varepsilon'|\int_{\mathbb{R}}|z|\varrho(z) \left(\int_{\mathbb{R}}|\partial_x^{l+1}f(x-s\varepsilon z-(1-s)\varepsilon'z)|^2
	\, \diff x \right)^{\frac{1}{2}} \,\diff z\\
	& \leq|\varepsilon-\varepsilon'| \|\partial_x^{l+1}f\|_{L^2(\mathbb{R})}\int_{\mathbb{R}}|z|\varrho(z) \,\diff z, 
\end{align*}
from which the estimate follows.
\end{proof}

%

\begin{proposition}  For any $0 < \varepsilon \ll 1$, \(N \geq 2\) and \(T_1 > 0\) as in \eqref{eq:assume v} the regularised problem \eqref{eq:regularised} has a unique  solution $U_\varepsilon\in C^1([0,T_1];H^N(\mathbb{R}))$.
\end{proposition}
\begin{proof}  We  express \eqref{eq:regularised} as an ODE in the Hilbert space $H^N(\mathbb{R})$:
\begin{align*}
\partial_tU_\varepsilon &=F(U_\varepsilon), \qquad U_\varepsilon(0,x)=U_0(x),
\end{align*}
with
\begin{align*}
F(U_\varepsilon)&=-\mathcal{J}_\varepsilon [\mathcal{J}_\varepsilon(A(V))\partial_x(\mathcal{J}_\varepsilon U_\varepsilon)]-\mathcal{J}_\varepsilon[\mathcal{J}_\varepsilon(B(V))\mathcal{K}\partial_x(\mathcal{J}_\varepsilon U_\varepsilon)]\nonumber\\
&=:F_1(U_\varepsilon)+F_2(U_\varepsilon).
\end{align*}

We first show that the map $F$ is bounded from $H^N(\mathbb{R})$ to $H^N(\mathbb{R})$. Because of the constant term appearing in \(\varphi + 2 \bar \lambda\) we shall use homogeneous estimates, in particular the following tame product estimate (cf. \cite{taylor1997partial})
 \begin{equation}\label{eq:tame k}
 \|\partial_x^k(fg)\|_{L^2(\mathbb{R})}\lesssim \|f\|_{L^\infty(\mathbb{R})}\|\partial_x^kg\|_{L^2(\mathbb{R})}+\|g\|_{L^\infty(\mathbb{R})}\|\partial_x^kf\|_{L^2(\mathbb{R})},
 \end{equation}
valid for integers $k\geq0$.
Thus, it follows from Lemma \ref{lemma:mollification} that
\begin{equation}\label{10.2}
\begin{aligned}
 &\|\partial_x^{N+1}(\mathcal{J}_\varepsilon(A(V))\partial_x(\mathcal{J}_\varepsilon U_\varepsilon))\|_{L^2(\mathbb{R})}\\
&\lesssim\| A(V)\|_{L^\infty(\mathbb{R})}\|\partial_x (\mathcal{J}_\varepsilon U_\varepsilon) \|_{H^{N+1}(\mathbb{R})}
+\|\partial_xU_\varepsilon\|_{L^\infty(\mathbb{R})}\|\partial_x^{N+1}\mathcal{J}_\varepsilon(A(V))\|_{L^2(\mathbb{R})}
\end{aligned}
\end{equation}
and
\begin{equation}
\begin{aligned}\label{10.4}
 &\|\mathcal{J}_\varepsilon(A(V))\partial_x(\mathcal{J}_\varepsilon U_\varepsilon)\|_{L^2(\mathbb{R})}
\lesssim \|A(V)\|_{L^\infty(\mathbb{R})}\|\partial_x(\mathcal{J}_\varepsilon U_\varepsilon)\|_{L^2(\mathbb{R})}\\
&\lesssim\| A(V)\|_{L^\infty(\mathbb{R})}\|\partial_x (\mathcal{J}_\varepsilon U_\varepsilon) \|_{H^{N+1}(\mathbb{R})}
+\|\partial_xU_\varepsilon\|_{L^\infty(\mathbb{R})}\|\partial_x^{N+1}\mathcal{J}_\varepsilon(A(V))\|_{L^2(\mathbb{R})}.
\end{aligned}
\end{equation}
Recall now the standard Gagliardo--Nirenberg interpolation inequality (see \cite{taylor1997partial}),
\begin{equation}\label{10.6}
\|\partial_x^lf\|_{L^2(\mathbb{X})}\leq C\|f\|_{L^2(\mathbb{X})}^{1-\frac{l}{k}}\|\partial_x^kf\|_{L^2(\mathbb{X})}^{\frac{l}{k}}, \qquad 0 \leq l \leq k.
\end{equation}
In view of \eqref{10.2}-\eqref{10.6} combined with the assumption \eqref{eq:assume v} on $V$ we obtain, using Lemma \ref{lemma:mollification}, that
\begin{equation}
\begin{aligned}\label{10.8}
&\|F_1(U_\varepsilon)\|_{H^N(\mathbb{R})} \\
&\lesssim \sum_{i=0}^N\| \mathcal{J}_\varepsilon(A(V))\partial_x(\mathcal{J}_\varepsilon U_\varepsilon)\|_{L^2(\mathbb{R})}^{1-\frac{i}{N+1}}\|\partial_x^{N+1}(\mathcal{J}_\varepsilon(A(V))\partial_x(\mathcal{J}_\varepsilon U_\varepsilon))\|_{L^2(\mathbb{R})}^{\frac{i}{N+1}}\\
&\lesssim \| A(V)\|_{L^\infty(\mathbb{R})}\|\partial_x(\mathcal{J}_\varepsilon U_\varepsilon)\|_{H^{N+1}(\mathbb{R})}
+\|\partial_xU_\varepsilon\|_{L^\infty(\mathbb{R})}\|\partial_x^{N+1}\mathcal{J}_\varepsilon(A(V))\|_{L^2(\mathbb{R})}\\
&\lesssim (\|V\|_{H^1(\mathbb{R})}+\bar\lambda)(\varepsilon^{-2}\|U_\varepsilon\|_{H^N(\mathbb{R}}))
+\|U_\varepsilon\|_{H^2(\mathbb{R})}(\varepsilon^{-1}\|V\|_{H^N(\mathbb{R})})\\
&\lesssim \varepsilon^{-2} (\mu^{-1} + E_N(U_0)^{\frac{1}{2}})\|U_\varepsilon\|_{H^N(\mathbb{R})}.
\end{aligned}
\end{equation}

Notice that since $\tanh(|\xi|)\leq 1$, it holds that
\[
\|\mathcal{K}\partial_xf\|_{H^s(\mathbb{R})}^2=\int_{\mathbb{R}}\frac{\xi^2 \tanh^2(\xi)}{\xi^2} (1+\xi^2)^s|\hat{f}(\xi)|^2\, \diff \xi\leq\|f\|_{H^s(\mathbb{R})}^2.
\]
Similar to \eqref{10.2} and \eqref{10.4} one has
\begin{equation}
\begin{aligned}\label{11.2}
&\|\mathcal{J}_\varepsilon(B(V))\mathcal{K}\partial_x(\mathcal{J}_\varepsilon U_\varepsilon)\|_{L^2(\mathbb{R})}+\|\partial_x^{N+1}(\mathcal{J}_\varepsilon(B(V))\mathcal{K}\partial_x(\mathcal{J}_\varepsilon U_\varepsilon))\|_{L^2(\mathbb{R})}\\
&\lesssim\| B(V)\|_{L^\infty(\mathbb{R})}\|\mathcal{K}\partial_x(\mathcal{J}_\varepsilon U_\varepsilon)\|_{H^{N+1}(\mathbb{R})}
+\|\mathcal{K}\partial_xU_\varepsilon\|_{L^\infty(\mathbb{R})}\|\partial_x^{N+1}\mathcal{J}_\varepsilon(B(V))\|_{L^2(\mathbb{R})}.
\end{aligned}
\end{equation}
Thus, \eqref{11.2} and the assumptions \eqref{eq:assume v} on $V$  show that
\begin{equation}
\begin{aligned}\label{12}
&\|F_2(U_\varepsilon)\|_{H^N(\mathbb{R})}\\
&\lesssim \sum_{i=0}^N\| \mathcal{J}_\varepsilon(B(V))\mathcal{K}\partial_x(\mathcal{J}_\varepsilon U_\varepsilon)\|_{L^2(\mathbb{R})}^{1-\frac{i}{N+1}}\|\partial_x^{N+1}(\mathcal{J}_\varepsilon(B(V))\mathcal{K}\partial_x(\mathcal{J}_\varepsilon U_\varepsilon))\|_{L^2(\mathbb{R})}^{\frac{i}{N+1}}\\
&\lesssim \| B(V)\|_{L^\infty(\mathbb{R})} \|\mathcal{K}\partial_x(\mathcal{J}_\varepsilon U_\varepsilon)\|_{H^{N+1}(\mathbb{R})}
+\|\mathcal{K}\partial_xU_\varepsilon\|_{L^\infty(\mathbb{R})}\|\partial_x^{N+1}\mathcal{J}_\varepsilon(B(V))\|_{L^2(\mathbb{R})}\\
&\lesssim (\mu \varepsilon)^{-1}\|U_\varepsilon\|_{H^N(\mathbb{R})}
+ (\mu \varepsilon)^{-N-1}\|U_\varepsilon\|_{H^1(\mathbb{R})}\sum_{i=0}^{N}{E_{N}(t,V)^{\frac{i}{2}}} \\
&\lesssim (\mu \varepsilon)^{-N-1} \|U_\varepsilon\|_{H^N(\mathbb{R})},
\end{aligned}
\end{equation}
implying the \(H^N\)-continuity of \(F\) (here and in the following paragraph we have suppressed the dependence on \(E_N(U_0))\). Because \(F\) is linear in \(U\) it is also locally Lipschitz  continuous  on any open set of $H^N(\mathbb{R})$, with the same estimates as above:
\[
\|F(U_\varepsilon^1)-F(U_\varepsilon^2)\|_{H^N(\mathbb{R})}
\lesssim (\mu \varepsilon)^{-N-1}  \|U_\varepsilon^1-U_\varepsilon^2\|_{H^N(\mathbb{R})}.
\]
Therefore, for any initial data $U_0\in H^N(\mathbb{R})$, Picard's theorem implies the existence of a positive time $T_\varepsilon$ and a unique solution $U_\varepsilon\in C^1([0,T_\varepsilon];H^N(\mathbb{R}))$ of the regularised problem \eqref{eq:regularised}.

Finally, we need only to show an a priori bound of
$\|U_\varepsilon(t,\cdot)\|_{H^N(\mathbb{R})}$ on $[0,T_1]$, which makes sure
that $T_\varepsilon$ can be extended to $T_1$. In fact,  it follows from \eqref{eq:regularised} that
\[
\frac{d}{dt}\|U_\varepsilon\|_{H^N(\mathbb{R})} \lesssim (\mu \varepsilon)^{-N-1}  \|U_\varepsilon\|_{H^N(\mathbb{R})},
\]
for all \(t \in [0,T_1]\). Combined with Gr\"onwall's inequality this gives
\[
\|U_\varepsilon(t,\cdot)\|_{H^N(\mathbb{R})}\lesssim 1 \quad \text{ for all } t\in[0,T_1],
\]
and where the estimate may grow, exponentially, in \((\mu \varepsilon)^{-N-1} \).
\end{proof}

\section{The linearised problem}\label{sec:linearised}
In this section we develop a priori estimates enabling us to take a limit in the regularised equation \eqref{eq:regularised}, thereby solving the linearised problem
\begin{equation}\label{eq:linearised}
\partial_tU+A(V)\partial_xU+B(V)\mathcal{K}\partial_xU=0,
\end{equation}
with \(U(0,x)=U_0(x)\). The main estimates appear in the proof of the following result.

\begin{proposition}\label{prop:linearised}  For any \(N \geq 2\) and any \(\mu\) as in \eqref{eq:assume u} and \eqref{eq:assume v} there exist a positive number $T_2$ and a unique  solution $U\in C([0,T_2];H^N(\mathbb{R}))\cap C^1([0,T_2];H^{N-1}(\mathbb{R}))$ of \eqref{eq:linearised} that satisfies
\begin{eqnarray*}
\max_{0\leq t\leq T_2}E_N(t,U)\leq2E_N(U_0),
\end{eqnarray*}
where the above norms of \(U\) for a fixed \(N\) depend only on \(\mu\) and \(E_N(U_0)\).
\end{proposition}

\begin{proof}
We apply $\partial_x^k$, $0\leq k\leq N$,  to \eqref{eq:regularised} and get
\begin{align*}
\partial_tU_\varepsilon^{(k)} &+\sum_{l=0}^kC_k^l\mathcal{J}_\varepsilon [ \mathcal{J}_\varepsilon(A(V^{(l)}))\partial_x(\mathcal{J}_\varepsilon U_\varepsilon^{(k-l)})]\\
& +\sum_{l=0}^kC_k^l\mathcal{J}_\varepsilon [\mathcal{J}_\varepsilon(B^{(l)}(V))\mathcal{K}\partial_x(\mathcal{J}_\varepsilon U_\varepsilon^{(k-l)})]=0.
\end{align*}
Thus
\begin{equation}\label{eq:dt Ek}
\begin{aligned}
\frac{1}{2}\frac{d}{dt}E^{(k)}(t,U_\varepsilon)
&=
-\sum_{l=0}^kC_k^l( \mathcal{J}_\varepsilon(A(V^{(l)}))\partial_x(\mathcal{J}_\varepsilon U_\varepsilon^{(k-l)}), \mathcal{J}_\varepsilon U_\varepsilon^{(k)})_2\\
&-\sum_{l=0}^kC_k^l(\mathcal{J}_\varepsilon(B^{(l)}(V))\mathcal{K}\partial_x(\mathcal{J}_\varepsilon U_\varepsilon^{(k-l)}), \mathcal{J}_\varepsilon U_\varepsilon^{(k)})_2.
\end{aligned}
\end{equation}
We start by investigating the terms in \eqref{eq:dt Ek} that include the matrix \(A(V)\). Because of the affine shift \(\varphi + 2 \bar\lambda\) appearing in the equation we consider separately the cases \(l=0\) and $1\leq l\leq k$. When $l=0$, the symmetry of $A(V)$ and integration by parts imply that
\begin{align*}
-(\mathcal{J}_\varepsilon(A(V))\partial_x(\mathcal{J}_\varepsilon U_\varepsilon^{(k)}), \mathcal{J}_\varepsilon U_\varepsilon^{(k)})_2
&=\frac{1}{2}(\mathcal{J}_\varepsilon(\partial_x[A(V)])\mathcal{J}_\varepsilon U_\varepsilon^{(k)}, \mathcal{J}_\varepsilon U_\varepsilon^{(k)})_2\\
&\lesssim \|V\|_{H^2(\mathbb{R})}\|U_\varepsilon^{(k)}\|_{L^2(\mathbb{R})}^2.
\end{align*}
On the other hand, when $1\leq l\leq k$ one has a derivative on the matrix \(A(V)\) to annihilate the constant in \(\varphi + 2 \bar\lambda\). For $1\leq l\leq k-1$, one has 
\begin{align*}
	&(\mathcal{J}_\varepsilon(A(V^{(l)}))\partial_x(\mathcal{J}_\varepsilon U_\varepsilon^{(k-l)}), \mathcal{J}_\varepsilon U_\varepsilon^{(k)})_2\\
	&\lesssim \|A(V^{(l)})\|_{L^\infty(\mathbb{R})}\|\partial_x( U_\varepsilon^{(k-l)})\|_{L^2(\mathbb{R})}\|U_\varepsilon^{(k)}\|_{L^2(\mathbb{R})}\\
	&\lesssim \|V\|_{H^k(\mathbb{R})}\|U_\varepsilon\|_{H^k(\mathbb{R})}\|U_\varepsilon^{(k)}\|_{L^2(\mathbb{R})},
\end{align*}
and, for $l=k$,
\begin{align*}
&(\mathcal{J}_\varepsilon(A(V^{(l)}))\partial_x(\mathcal{J}_\varepsilon U_\varepsilon^{(k-l)}), \mathcal{J}_\varepsilon U_\varepsilon^{(k)})_2\\
&\lesssim \|A(V^{(l)})\|_{L^2(\mathbb{R})}\|\partial_xU_\varepsilon\|_{L^\infty(\mathbb{R})}\|U_\varepsilon^{(k)}\|_{L^2(\mathbb{R})}\\
&\lesssim \|V\|_{H^k(\mathbb{R})}\|U_\varepsilon\|_{H^2(\mathbb{R})}\|U_\varepsilon^{(k)}\|_{L^2(\mathbb{R})}.
\end{align*}
Now, as what concerns the terms in \eqref{eq:dt Ek} including \(B(V)\), note first that
\begin{align*}
&(\mathcal{J}_\varepsilon(B^{(l)}(V))\mathcal{K}\partial_x(\mathcal{J}_\varepsilon U_\varepsilon^{(k-l)}), \mathcal{J}_\varepsilon U_\varepsilon^{(k)})_2\\ 
&=2\int_{\mathbb{R}}\mathcal{J}_\varepsilon \left(\frac{1}{\varphi+2\bar\lambda}\right)^{(l)}\mathcal{K}\partial_x(\mathcal{J}_\varepsilon u_\varepsilon^{(k-l)})\mathcal{J}_\varepsilon\zeta_\varepsilon^{(k)} \, \diff x.
\end{align*}
The case $l=0$ is straightforward, as
\begin{align*}
\int_{\mathbb{R}}\mathcal{J}_\varepsilon \left(\frac{1}{\varphi+2\bar\lambda} \right)\mathcal{K}\partial_x(\mathcal{J}_\varepsilon u_\varepsilon^{(k)})\mathcal{J}_\varepsilon \zeta_\varepsilon^{(k)} \, \diff x
&\lesssim \mu^{-1} \|\mathcal{K}\partial_xu_\varepsilon^{(k)}\|_{L^2(\mathbb{R})}\|\zeta_\varepsilon^{(k)}\|_{L^2(\mathbb{R})}\\
&\lesssim \mu^{-1}\|u_\varepsilon^{(k)}\|_{L^2(\mathbb{R})}\|\zeta_\varepsilon^{(k)}\|_{L^2(\mathbb{R})}.
\end{align*}
On the other hand, when $1\leq l\leq k$, Leibniz's rule and the assumptions \eqref{eq:assume v} on $V$ yield that
\begin{align*}
\|\mathcal{J}_\varepsilon(({\textstyle \frac{1}{\varphi+2\bar\lambda}})^{(l)})\|_{L^2(\mathbb{R})}&\lesssim \mu^{-l} (\|\varphi^{(l)}\|_{L^2(\mathbb{R})}+\cdots+\|\partial_x\varphi\|_{L^2(\mathbb{R})}\|\partial_x\varphi\|_{L^\infty(\mathbb{R})}^{(l-1)})\\
&\leq \mu^{-l} \sum_{i=1}^N(2E_N(U_0))^{\frac{i}{2}}.
\end{align*}
For the same range of $l$, we thus deduce that
\begin{align*}
&\int_{\mathbb{R}}\mathcal{J}_\varepsilon  \left(\frac{1}{\zeta+2\bar\lambda}\right)^{(l)} \mathcal{K}\partial_x(\mathcal{J}_\varepsilon u_\varepsilon^{(k-l)})\mathcal{J}_\varepsilon\zeta_\varepsilon^{(k)} \, \diff x\\
&\lesssim \sum_{i=1}^N(2E_N(U_0))^{\frac{i}{2}} \|\mathcal{K}\partial_xu_\varepsilon^{(k-l)}\|_{L^\infty(\mathbb{R})}\|\zeta_\varepsilon^{(k)}\|_{L^2(\mathbb{R})}\\
&\lesssim \sum_{i=1}^N(2E_N(U_0))^{\frac{i}{2}} \|u_\varepsilon^{(k-l)}\|_{H^1(\mathbb{R})}\|\zeta_\varepsilon^{(k)}\|_{L^2(\mathbb{R})},
\end{align*}
where we have now suppressed the dependence on \(\mu\), as it is fixed and from the above estimates clearly controlled. Therefore, the \(B\)-part of \eqref{eq:dt Ek} may be controlled as
\begin{align*}
&-\sum_{l=0}^kC_k^l(\mathcal{J}_\varepsilon(B^{(l)}(V))\mathcal{K}\partial_x(\mathcal{J}_\varepsilon U_\varepsilon^{(k-l)}), \mathcal{J}_\varepsilon U_\varepsilon^{(k)})_2\\
&\lesssim (1+\sum_{i=1}^N(2E_N(U_0))^{\frac{i}{2}}) \|u_\varepsilon\|_{H^k(\mathbb{R})}\|\zeta_\varepsilon\|_{H^k(\mathbb{R})}.
\end{align*}
We conclude from this and the above estimates for the \(A\)-part that, in total,
\begin{equation}\label{eq:dt Ek bound}
\frac{d}{dt}E^{(k)}(t,U_\varepsilon)\lesssim (1+\sum_{i=1}^N(2E_N(U_0))^{\frac{i}{2}}) ( E^{(k)}(t,U_\varepsilon) )^{\frac{1}{2}} ( E_N (t,U_\varepsilon) )^{\frac{1}{2}},
\end{equation}
where the estimate is uniform in $\varepsilon$. Summing over $k$ from $0$ to $N$ gives
\begin{equation}\label{eq:dt EN bound}
\frac{d}{dt}E_N(t,U_\varepsilon)\lesssim (1+\sum_{i=1}^N(2E_N(U_0))^{\frac{i}{2}}) E_N(t,U_\varepsilon),
\end{equation}
and Gr\"onwall's inequality now guarantees the existence of 
\[
T_2 \eqsim \min\left(T_1,\frac{\ln2}{(1+\sum_{i=1}^N(2E_N(U_0))^{\frac{i}{2}})}\right)
\] 
such that
\begin{equation}\label{eq:a priori energy}
\max_{0\leq t\leq T_2}E_N(t,U_\varepsilon)\leq2E_N(U_0).
\end{equation}
The family $\{U_\varepsilon\}_{\varepsilon}$ is therefore uniformly bounded in $C([0, T_2]; H^N(\mathbb{R}))$.

{\bf Convergence.} We shall now prove that a subsequence of the family $\{U_\varepsilon\}_\varepsilon$ defines a Cauchy sequence in $C({[0,T_2]; L^2 (\mathbb{R})})$. It follows from  \eqref{eq:regularised} that the difference $U_\varepsilon-U_{\varepsilon'}$ of two solutions of the regularised problem satisfies 
\begin{align*}
\partial_t(U_\varepsilon-U_{\varepsilon'})+\mathcal{J}_\varepsilon[\mathcal{J}_\varepsilon(A(V))\partial_x(\mathcal{J}_\varepsilon U_\varepsilon)]-\mathcal{J}_{\varepsilon'}[\mathcal{J}_{\varepsilon'}(A(V))\partial_x(\mathcal{J}_{\varepsilon'} U_{\varepsilon'})]&\\
+\mathcal{J}_\varepsilon[\mathcal{J}_\varepsilon(B(V))\mathcal{K}\partial_x(\mathcal{J}_\varepsilon U_\varepsilon)]-\mathcal{J}_{\varepsilon'}[\mathcal{J}_{\varepsilon'}(B(V))\mathcal{K}\partial_x(\mathcal{J}_{\varepsilon'} U_{\varepsilon'})]&=0,
\end{align*}
on $[0, T_2]\times\mathbb{R}$, while additionally having vanishing initial data \((U_\varepsilon-U_{\varepsilon'})|_{t=0}=0\).
Therefore,
\begin{align*}
&\frac{1}{2}\frac{d}{dt}\|U_\varepsilon-U_{\varepsilon'}\|_{L^2(\mathbb{R})}^2\\
&=-(\mathcal{J}_\varepsilon[\mathcal{J}_\varepsilon(A(V))\partial_x(\mathcal{J}_\varepsilon U_\varepsilon)]-\mathcal{J}_{\varepsilon'}[\mathcal{J}_{\varepsilon'}(A(V))\partial_x(\mathcal{J}_{\varepsilon'} U_{\varepsilon'})],
U_\varepsilon-U_{\varepsilon'})_2\\
&\quad-(\mathcal{J}_\varepsilon[\mathcal{J}_\varepsilon(B(V))\mathcal{K}\partial_x(\mathcal{J}_\varepsilon U_\varepsilon)]-\mathcal{J}_{\varepsilon'}[\mathcal{J}_{\varepsilon'}(B(V))\mathcal{K}\partial_x(\mathcal{J}_{\varepsilon'} U_{\varepsilon'})],U_\varepsilon-U_{\varepsilon'})_2\\
&=:I+J.
\end{align*}
We split the  terms $I$ and $J$ into the additional parts
\begin{align*}
I
&=-(\mathcal{J}_\varepsilon[\mathcal{J}_\varepsilon(A(V))\partial_x(\mathcal{J}_\varepsilon (U_\varepsilon-U_{\varepsilon'}))],
U_\varepsilon-U_{\varepsilon'})_2\\
&-(\mathcal{J}_\varepsilon[\mathcal{J}_\varepsilon(A(V))\partial_x((\mathcal{J}_\varepsilon-\mathcal{J}_{\varepsilon'})U_{\varepsilon'})],
U_\varepsilon-U_{\varepsilon'})_2\\
&-(\mathcal{J}_\varepsilon[(\mathcal{J}_\varepsilon-\mathcal{J}_{\varepsilon'})(A(V))\partial_x(\mathcal{J}_{\varepsilon'} U_{\varepsilon'})],
U_\varepsilon-U_{\varepsilon'})_2\\
&-((\mathcal{J}_\varepsilon-\mathcal{J}_{\varepsilon'})[\mathcal{J}_{\varepsilon'}(A(V))\partial_x(\mathcal{J}_{\varepsilon'} U_{\varepsilon'})],
U_\varepsilon-U_{\varepsilon'})_2\\
&=:I_1+I_2+I_3+I_4,
\end{align*}
and
\begin{align*}
J &=-(\mathcal{J}_\varepsilon[\mathcal{J}_\varepsilon(B(V))\mathcal{K}\partial_x(\mathcal{J}_\varepsilon (U_\varepsilon-U_{\varepsilon'}))],
U_\varepsilon-U_{\varepsilon'})_2\\
&-(\mathcal{J}_\varepsilon[\mathcal{J}_\varepsilon(B(V))\mathcal{K}\partial_x((\mathcal{J}_\varepsilon-\mathcal{J}_{\varepsilon'})U_{\varepsilon'})],
U_\varepsilon-U_{\varepsilon'})_2\\
&-(\mathcal{J}_\varepsilon[(\mathcal{J}_\varepsilon-\mathcal{J}_{\varepsilon'})(B(V))\mathcal{K}\partial_x(\mathcal{J}_{\varepsilon'} U_{\varepsilon'})],
U_\varepsilon-U_{\varepsilon'})_2\\
&-((\mathcal{J}_\varepsilon-\mathcal{J}_{\varepsilon'})[\mathcal{J}_{\varepsilon'}(B(V))\mathcal{K}\partial_x(\mathcal{J}_{\varepsilon'} U_{\varepsilon'})],
U_\varepsilon-U_{\varepsilon'})_2\\
&=: J_1+J_2+J_3+J_4.
\end{align*}
These eight terms may be estimated as follows. For \(I_1\) integration by parts yields
\begin{align*}
I_1
&=\frac{1}{2}(\mathcal{J}_\varepsilon(\partial_x[A(V)])\mathcal{J}_\varepsilon (U_\varepsilon-U_{\varepsilon'}),
\mathcal{J}_\varepsilon(U_\varepsilon-U_{\varepsilon'}))_2\\
&\lesssim \|\partial_x[A(V)]\|_{L^\infty(\mathbb{R})}\|U_\varepsilon-U_{\varepsilon'}\|_{L^2(\mathbb{R})}^2\\
&\lesssim E_N(U_0)^\frac{1}{2}\|U_\varepsilon-U_{\varepsilon'}\|_{L^2(\mathbb{R})}^2,
\end{align*}
and it is easy to see that
\begin{eqnarray*}
J_1\lesssim \|B(V)\|_{L^{\infty}} \|U_\varepsilon-U_{\varepsilon'}\|_{L^2(\mathbb{R})}^2\lesssim \mu^{-1} \|U_\varepsilon-U_{\varepsilon'}\|_{L^2(\mathbb{R})}^2.
\end{eqnarray*}
It follows from Lemma \ref{lemma:mollification} that
\begin{align*}
I_2&\lesssim |\varepsilon -\varepsilon' | \|A(V)\|_{L^\infty(\mathbb{R})}\|U_{\varepsilon'}\|_{H^2(\mathbb{R})}\|U_\varepsilon-U_{\varepsilon'}\|_{L^2(\mathbb{R})}\\
&\lesssim |\varepsilon -\varepsilon' | (E_N(U_0)^\frac{1}{2}+\bar{\lambda})E_N(U_0)^\frac{1}{2}\|U_\varepsilon-U_{\varepsilon'}\|_{L^2(\mathbb{R})},
\end{align*}
and
\begin{eqnarray*}
J_2&\lesssim& |\varepsilon -\varepsilon' | \|B(V)\|_{L^\infty(\mathbb{R})}\|U_{\varepsilon'}\|_{H^1(\mathbb{R})}\|U_\varepsilon-U_{\varepsilon'}\|_{L^2(\mathbb{R})}\nonumber\\
&\lesssim& |\varepsilon -\varepsilon' | \mu^{-1} E_N(U_0)^\frac{1}{2}\|U_\varepsilon-U_{\varepsilon'}\|_{L^2(\mathbb{R})}.
\end{eqnarray*}
To estimate \(I_3\) and \(J_3\), write
\[
A(U)=
  \begin{pmatrix}
    u & \frac{\zeta}{2} \\
    \frac{\zeta}{2} & u
  \end{pmatrix}
+
  \begin{pmatrix}
    0 & \bar\lambda \\
    \bar\lambda & 0
  \end{pmatrix}
=:A_1(U)+
  \begin{pmatrix}
    0 & \bar\lambda \\
    \bar\lambda & 0
  \end{pmatrix}.
\]
Then Lemma \ref{lemma:mollification} implies that
\begin{align*}
 \|(\mathcal{J}_\varepsilon-\mathcal{J}_{\varepsilon'})(A(V))\|_{L^\infty(\mathbb{R})} &\lesssim \|(\mathcal{J}_\varepsilon-\mathcal{J}_{\varepsilon'})(A_1(V))\|_{H^1(\mathbb{R})}  \lesssim |\varepsilon -\varepsilon'| E_N(U_0)^\frac{1}{2}.
\end{align*}
Similarly, by the assumption \eqref{eq:assume v} on $V$, we obtain
\[
\|(\mathcal{J}_\varepsilon- \mathcal{J}_{\varepsilon'})(B(V))\|_{L^\infty(\mathbb{R})} \lesssim \mu^{-2} \|(\mathcal{J}_\varepsilon- \mathcal{J}_{\varepsilon'})\varphi\|_{L^\infty(\mathbb{R})},
\]
which via Lemma~\ref{lemma:mollification} leads to
\[
 \|(\mathcal{J}_\varepsilon-\mathcal{J}_{\varepsilon'})(B(V))\|_{L^\infty(\mathbb{R})} \lesssim  |\varepsilon -\varepsilon'| \mu^{-2} E_N(U_0)^\frac{1}{2}.
\]
One thus obtains
\begin{align*}
I_3&\lesssim \|(\mathcal{J}_\varepsilon-\mathcal{J}_{\varepsilon'})(A(V))\|_{L^\infty(\mathbb{R})}\|U_{\varepsilon'}\|_{H^1(\mathbb{R})}\|U_\varepsilon-U_{\varepsilon'}\|_{L^2(\mathbb{R})}\\
&\lesssim |\varepsilon - \varepsilon' | E_N(U_0)\|U_\varepsilon-U_{\varepsilon'}\|_{L^2(\mathbb{R})}
\end{align*}
and
\begin{eqnarray*}
J_3
&\lesssim&\|(\mathcal{J}_\varepsilon-\mathcal{J}_{\varepsilon'})(B(V))\|_{L^\infty(\mathbb{R})}\|U_{\varepsilon'}\|_{L^2(\mathbb{R})}\|U_\varepsilon-U_{\varepsilon'}\|_{L^2(\mathbb{R})}\nonumber\\
&\lesssim& |\varepsilon - \varepsilon'| \mu^{-2} E_N(U_0)\|U_\varepsilon-U_{\varepsilon'}\|_{L^2(\mathbb{R})}.
\end{eqnarray*}
For the last two terms we have, in analogy with \eqref{10.2}--\eqref{10.4}, that
\begin{align*}
I_4&\lesssim \|(\mathcal{J}_\varepsilon-\mathcal{J}_{\varepsilon'})[\mathcal{J}_{\varepsilon'}(A(V))\partial_x(\mathcal{J}_{\varepsilon'} U_{\varepsilon'})]\|_{L^2(\mathbb{R})}\|U_\varepsilon-U_{\varepsilon'}\|_{L^2(\mathbb{R})}\\
&\lesssim |\varepsilon - \varepsilon'| \big(\| A(V)\|_{L^\infty(\mathbb{R})}\|\partial_x^2 U_{\varepsilon'}\|_{L^2(\mathbb{R})}\\
&\qquad+\|\partial_xU_\varepsilon\|_{L^\infty(\mathbb{R})}\|\partial_x A(V)\|_{L^2(\mathbb{R})}\big)  \|U_\varepsilon-U_{\varepsilon'}\|_{L^2(\mathbb{R})}\\
&\lesssim |\varepsilon - \varepsilon'| (E_N(U_0)^\frac{1}{2}+\mu^{-1})E_N(U_0)^\frac{1}{2}\|U_\varepsilon-U_{\varepsilon'}\|_{L^2(\mathbb{R})}
\end{align*}
and
\begin{align*}
J_4&\lesssim \|(\mathcal{J}_\varepsilon-\mathcal{J}_{\varepsilon'})[\mathcal{J}_{\varepsilon'}(B(V))\mathcal{K}\partial_x(\mathcal{J}_{\varepsilon'} U_{\varepsilon'})]\|_{L^2(\mathbb{R})}\|U_\varepsilon-U_{\varepsilon'}\|_{L^2(\mathbb{R})}\\
&\lesssim |\varepsilon - \varepsilon'| \big( \| B(V)\|_{L^\infty(\mathbb{R})}\|\mathcal{K}\partial_x^2 U_{\varepsilon'}\|_{L^2(\mathbb{R})}\\
&\quad +\|\mathcal{K}\partial_xU_\varepsilon\|_{L^\infty(\mathbb{R})}\|\partial_x B(V)\|_2 \big)  \|U_\varepsilon-U_{\varepsilon'}\|_{L^2(\mathbb{R})}\\
&\lesssim |\varepsilon - \varepsilon'| (E_N(U_0)^\frac{1}{2}+\mu^{-1}) \mu^{-1}E_N(U_0)^\frac{1}{2} \|U_\varepsilon-U_{\varepsilon'}\|_{L^2(\mathbb{R})}.
\end{align*}
Therefore, we conclude that
\[
\frac{d}{dt}\|U_\varepsilon-U_{\varepsilon'}\|_{L^2(\mathbb{R})}
\lesssim_{\mu, E_N(U_0)} \|U_\varepsilon-U_{\varepsilon'}\|_{L^2(\mathbb{R})}+|\varepsilon -\varepsilon'|,
\]
which by Gr{\"o}nwall's inequality gives that
\[
\max_{0\leq t\leq T_2}\|U_\varepsilon-U_{\varepsilon'}\|_{L^2(\mathbb{R})} \lesssim |\varepsilon -\varepsilon'|,
\]
where the estimate is uniform with respective to \(t \in [0,T_2]\) and fixed values of \(\mu\) and \(E_N(U_0)\). Consequently, up to subsequences, the family $\{U_\varepsilon\}_\varepsilon$ converges in $C([0, T_2]; L^2(\mathbb{R}))$ to a pair $U=(\zeta,u)$ as $\varepsilon\searrow 0$.
By the standard interpolation inequality 
\begin{eqnarray}\label{13}
\|\zeta_\varepsilon-\zeta\|_{H^s(\mathbb{R})}\leq \|\zeta_\varepsilon-\zeta\|_{L^2(\mathbb{R})}^{1-\frac{s}{N}}
\|\zeta_\varepsilon-\zeta\|_{H^N(\mathbb{R})}^{\frac{s}{N}},
\end{eqnarray}
which is valid for all  $s \in (0,N)$, see \cite{taylor1997partial},  the family $\{\zeta_\varepsilon\}_\varepsilon$ converges to $\zeta$ in $C([0, T_2]; H^s(\mathbb{R}))$. Similarly,  $\{u_\varepsilon\}$ converges to $u$ in $C([0, T_2]; H^s(\mathbb{R}))$ for the same values of \(s\). Thus,  $\{U_\varepsilon\}_\varepsilon$ converges to $U$ in $ C([0, T_2]; C^1(\mathbb{R}))$ by Sobolev embedding.
One can furthermore deduce from equation \eqref{eq:linearised} that $\{\partial_tU_\varepsilon\}$ converges to $\partial_tU$ in $C([0, T_2]; C(\mathbb{R}))$.
Consequently, $U$ is a classical solution to \eqref{eq:linearised}. 

The following standard argument, built upon the a priori estimate \eqref{eq:a priori energy} and the time reversibility of \eqref{eq:linearised}, shows that $U$ is also a unique classical solution in $C([0, T_2]; H^N(\mathbb{R}))\cap C^1([0, T_2]; H^{N-1}(\mathbb{R}))$. By the a priori estimate $E_N(U_\varepsilon)\leq 2E_N(U_0)$, one sees that the family  \(\{U_\varepsilon\}_\varepsilon\)  is uniformly bounded in \(L^2([0,T_2]; H^N(\mathbb{R}))\). By weak compactness, there thus exists a subsequence $\{U_{\varepsilon_j}\}_j \subset \{U_\varepsilon\}_\varepsilon$ such that
\[
U_{\varepsilon_j} \rightharpoonup U \qquad\text{ in } \quad L^2([0,T_2]; H^N(\mathbb{R})),
\]
as \(\varepsilon_j \searrow 0\). Furthermore, for each fixed time $t\in [0,T_2]$,  one can pick up a new subsequence $\{U_{\varepsilon_{j_l}}(t,\cdot)\}_{j_l} \subset \{U_{\varepsilon_j}(t,\cdot)\}_j $ such that 
\[
U_{\varepsilon_{j_l}}(t,\cdot) \rightharpoonup U(t,\cdot) \qquad\text{ in } \quad H^N(\mathbb{R}), 
\]	
as \(\varepsilon_{j_l} \searrow 0\), wherefore it holds that
\[
\sup_{t \in [0,T_2]} \|U(t,\cdot)\|_{H^N(\mathbb{R})} \leq \sup_{t \in [0,T_2]}  \liminf_{\varepsilon_{j_l} \to 0} \|U_{\varepsilon_{j_l}} (t,\cdot)\|_{H^N(\mathbb{R})} \leq 2 E_n(U_0),
\]
and $U\in L^\infty([0, T_2]; H^N(\mathbb{R}))$. When \(s \in (0,N)\), we have from \eqref{13} the stronger statement that $U_\varepsilon\rightarrow U$ in  $C([0,T_2]; H^s(\mathbb{R}))$. Now, pick a  test function $\phi\in H^{-s}(\mathbb{R})$. Then 
\begin{equation}\label{eq:H-s convergence}
(U_\varepsilon,\phi)_2(t)\rightarrow (U,\phi)_2(t) \qquad \text{uniformly for }  t \in [0, T_2]. 
\end{equation}
By further using that \(U\) and \(\{U_\varepsilon\}_\varepsilon\) are bounded in \(L^\infty([0, T_2]; H^N(\mathbb{R}))\) and that the embedding $H^{-s}(\mathbb{R}) \hookrightarrow H^{-N}(\mathbb{R})$ is dense, one finds that \eqref{eq:H-s convergence} holds also for $\phi\in H^{-N}(\mathbb{R})$.
In effect,
\[
\| U_0 \|_{H^N(\mathbb{R})}\leq \liminf_{t\downarrow 0} \| U(t,\cdot) \|_{H^N(\mathbb{R})}. 
\]
But from \eqref{eq:dt EN bound} we also have 
\[
\sup_{0\leq \tau\leq t}E_N(U_\varepsilon)\leq E_N(U_0)+ c_{E_N(U_0)} \int_0^t E_N(U_\varepsilon) \, \diff s,  \qquad t \in [0,T_2],
\]
whence $\limsup_{t\downarrow 0} \|U(t,\cdot)\|_{H^N(\mathbb{R})}\leq \|U_0\|_{H^N(\mathbb{R})}$. Thus 
\[
\lim_{t\downarrow 0} \|U(t,\cdot)\|_{H^N(\mathbb{R})} = \|U_0\|_{H^N(\mathbb{R})}. 
\]
Since equation \eqref{eq:linearised} is time reversible, \(\| U(t,\cdot) \|_{H^N(\R)}\) is in fact continuous at \(t=0\), with limit \( \|U_0 \|_{H^N(\mathbb{R})}\). Then, for any $t^*\in[0,T_2)$ one may regard $U(t^*,\cdot)$ as new initial data and re-solve the equations, to prove that $\|U(t,\cdot)\|_{H^N(\mathbb{R})}$ is everywhere continuous in \([0,T_2]\) (where we consider only continuity from the left at the endpoint \(t = T_2\). Finally, one finds $U\in C^1([0, T_2]; H^{N-1}(\mathbb{R}))$ from \eqref{eq:linearised} by using that $U\in C([0, T_2]; H^N(\mathbb{R}))$.

The uniqueness can then be easily proved as follows. Let $U$ and $\tilde{U}$ be two solutions of \eqref{eq:linearised} with the same initial data $U_0$, and let  $W=U-\tilde{U}$. Then
\begin{align*}
\partial_tW+A(V)\partial_xW+B(V)\mathcal{K}\partial_xW &=0,\\
W(0,x) &= 0,
\end{align*}
on $[0, T_2]\times\mathbb{R}$. Direct calculation shows that
\begin{align*}
{\textstyle \frac{1}{2}\frac{d}{dt}} \|W\|_{L^2(\mathbb{R})}^2 &=(\partial_x(A(V))W, W)_2-(B(V)\mathcal{K}\partial_xW, W)_2\\
&\lesssim  (E_N(U_0)^{\frac{1}{2}}+\mu^{-1})\|W\|_{L^2(\mathbb{R})}^2,
\end{align*}
which together with Gr{\"o}nwall's inequality yields that
\begin{align*}
W(t,x) = 0,\qquad (t,x) \in [0, T_2]\times\mathbb{R}.
\end{align*}
This concludes the proof.
\end{proof}

\section{Proof of the main theorem}\label{proof of main theorem}
In this section we give the proof of the main result on the line. A setup sufficient to follow the same procedure in the periodic case is described in the appendix.

\subsection{The case of the line, \(\mathbb{X} = \R\)}
We note first the following lemma, which is immediate from the uniform bound on \(\|\partial_t U(t,\cdot) \|_{H^N(\R)}\) proved in Proposition~\ref{prop:linearised}. 

\begin{lemma}\label{lemma:u gives v}
There exists \(T_3 \in (0, T_2]\), depending only on \(N\) and on \(\mu\), such that  if the initial data \(U_0\) satisfies \eqref{eq:assume u} then the assumption \eqref{eq:assume v} holds with  \(V\) replaced by $U$  on $[0, T_{3}]\times \R$, where $U$ is  the solution to the linearised equation \eqref{eq:linearised}.
\end{lemma}

We have now come to the proof of the main result. 
 
\begin{proof}[Proof of Theorem~\ref{thm:main}] 
With $U_0$ being our initial data, we consider the following  series of  linearised problems  for $m \in \N_0$.
\begin{equation}\label{14}
\begin{aligned}
\partial_tU_{m+1}+A(U_m)\partial_xU_{m+1}+B(U_m)\mathcal{K}\partial_xU_{m+1} &=0,\\
U_{m+1}(0,\cdot) &=U_0.
\end{aligned}
\end{equation}
Note that $u_0$ satisfies $\eqref{3}$, and that the positive constant \(\mu \leq \bar \lambda\) is chosen so that \eqref{eq:assume u} holds. By induction on $m$ and using Proposition~\ref{prop:linearised} and Lemma~\ref{lemma:u gives v}, for each $m$, there exists a solution $U_m\in C([0,T_3];H^N(\mathbb{R}))\cap C^1([0,T_3];H^{N-1}(\mathbb{R}))$ of \eqref{14} satisfying the assumption \eqref{eq:assume v} on \(V\) in the linearised equation \eqref{eq:linearised}. Therefore, for any $1\leq l \leq N$,
\begin{eqnarray*}
\Big\| \Big(\frac{1}{\zeta_m(t,\cdot)+2\bar\lambda}\Big)^{(l)}\Big\|_{L^2(\mathbb{R})} \lesssim_{\mu} \sum_{i=1}^NE_N(t,U_m)^{\frac{i}{2}}.
\end{eqnarray*}
We suppress now the dependence on \(\mu^{-1}\), since it is a fixed and bounded number. Similar to  \eqref{eq:dt EN bound}, we now have
\begin{eqnarray*}
\frac{d}{dt}E_N(t,U_{m+1})\lesssim (1+\sum_{i=1}^NE_N(t,U_m)^{\frac{i}{2}}) E_N(t,U_{m+1}),
\end{eqnarray*}
where the estimate is independent of $m$.
By induction on $m$, one has
\begin{eqnarray*}
\max_{0\leq t\leq T_3}E_N(t,U_m)\leq2E_N(U_0)\quad \text{ for all } m \in \N_0.
\end{eqnarray*}
The family $\{U_m\}$  is thus uniformly bounded in $C([0, T_3]; H^N(\mathbb{R}))$.

We shall now prove that  $\{U_m\}_{m}$ forms a Cauchy sequence in $C({[0,T_3]; L^2 (\mathbb{R})})$. For each \(m \geq 1\), let $W_{m+1}=U_{m+1}-U_m$. It then follows from \eqref{14} that 
\begin{align*}
\partial_tW_{m+1}+A(U_m)\partial_xW_{m+1}+B(U_m)\mathcal{K}\partial_xW_{m+1}\\
\quad+(A(U_m)-A_1(U_{m-1}))\partial_xU_m \:\:\\ 
\quad+(B(U_m)-B(U_{m-1}))\mathcal{K}\partial_xU_m &=0,\\
W_{m+1}(0,x) &= 0,
\end{align*}
again on $[0, T_3]\times\mathbb{R}$. Consequently, 
\begin{align*}
{\textstyle \frac{1}{2}\frac{d}{dt}} \|W_{m+1}\|_{L^2(\mathbb{R})}^2 &=-(A(U_m)\partial_xW_{m+1}, W_{m+1})_2\\
&\quad-(B(U_m)\mathcal{K}\partial_xW_{m+1}, W_{m+1})_2\\
&\quad-((A(U_m)-A(U_{m-1}))\partial_xU_m,W_{m+1})_2\\
&\quad -((B(U_m)-B(U_{m-1}))\mathcal{K}\partial_xU_m, W_{m+1})_2,
\end{align*}
and we may estimate the right-hand side as follows:
\begin{align*}
-(A(U_m)\partial_xW_{m+1}, W_{m+1})_2 &={\textstyle \frac{1}{2}} (\partial_x\left(A(U_m)\right)W_{m+1}, W_{m+1})_2\\
&\lesssim \|U_m\|_{H^2(\mathbb{R})}\|W_{m+1}\|_{L^2(\mathbb{R})}^2\\
&\lesssim E_N(U_0)^{\frac{1}{2}}\|W_{m+1}\|_{L^2(\mathbb{R})}^2,\\
-(B(U_m)\mathcal{K}\partial_xW_{m+1}, W_{m+1})_2 &\lesssim \|W_{m+1}\|_{L^2(\mathbb{R})}^2,\\[12pt]
-((A(U_m)-A(U_{m-1}))\partial_xU_m,W_{m+1})_2 &\lesssim \|W_m\|_{L^2(\mathbb{R})}\|U_m\|_{H^2(\mathbb{R})}\|W_{m+1}\|_{L^2(\mathbb{R})}\\
&\lesssim E_N(U_0)^{\frac{1}{2}}\|W_m\|_{L^2(\mathbb{R})}\|W_{m+1}\|_{L^2(\mathbb{R})}\\
\intertext{and}
-((B(U_m)-B(U_{m-1}))\partial_xU_m,W_{m+1})_2 &\lesssim \|W_m\|_{L^2(\mathbb{R})}\|U_m\|_{H^2(\mathbb{R})}\|W_{m+1}\|_{L^2(\mathbb{R})}\\
&\lesssim E_N(U_0)^{\frac{1}{2}}\|W_m\|_{L^2(\mathbb{R})}\|W_{m+1}\|_{L^2(\mathbb{R})},
\end{align*}
where again the dependence on \(\mu\) has been suppressed. We may thus conclude that
\begin{eqnarray*}
{\textstyle \frac{d}{dt}} \|W_{m+1}\|_{L^2(\mathbb{R})} \lesssim_{\mu,E_N(U_0)} \|W_{m+1}\|_{L^2(\mathbb{R})}+\|W_m\|_{L^2(\mathbb{R})}.
\end{eqnarray*}
By Gr{\"o}nwall's inequality,
\begin{eqnarray*}
\max_{0\leq t\leq T} \|W_{m+1}\|_{L^2(\mathbb{R})}
\lesssim_{\mu,E_N(U_0)} T\exp(c_{\mu,E_N(U_0)} T)\max_{0\leq t\leq T}\|W_m\|_{L^2(\mathbb{R})},
\end{eqnarray*}
and we may choose $T \leq T_3$  such that
\begin{eqnarray*}
\|W_{m+1}\|_{C([0,T];L^2(\mathbb{R}))}
\leq {\textstyle \frac{1}{2}}  \|W_m\|_{C([0,T];L^2(\mathbb{R}))}.
\end{eqnarray*}
This immediately implies that \(\{U_m\}_m\) is a Cauchy sequence in the same space, and there thus exists a pair $(\zeta,u)$ such that
\begin{eqnarray}\label{16}
\|\zeta_m-\zeta\|_{C([0,T];L^2(\mathbb{R}))}+\|u_m-u\|_{C([0,T];L^2(\mathbb{R}))}\rightarrow 0,
\end{eqnarray}
as \(m\rightarrow\infty\). In view of  \eqref{16}, similar to the end of the proof of Proposition~\ref{prop:linearised},  one can show that $U$ is a unique classical solution of \eqref{5} in the sense of \(C([0, T]; H^N(\mathbb{R})) \cap C^1([0, T]; H^{N-1}(\mathbb{R}))\). That the solution $U$ depends continuously on the initial data $U_0$ follows from a Bona--Smith type argument \cite{bona1975initial}, where we underline that the constant \(\bar \lambda\) is held fixed in this argument, and thus in the metric given by Theorem~\ref{thm:main}.
\end{proof}

\section*{Appendix}  
With the proof of Theorem~\ref{thm:main} completed in the case when \(\mathbb{X} = \R\), to establish the same result in the periodic case \(\mathbb{X} = \mathbb{T}\) requires only to set up a suitable functional framework, in which estimates from the line may be transferred to estimates on the torus. We provide in this appendix the necessary tools for such a procedure. In particular, we define the appropriate periodic Bessel-potential spaces and prove a periodic version of Lemma~\ref{lemma:mollification}. 
  
Let $\mathcal{E}(\mathbb{T})$ be the Fr\'echet space of all infinitely continuously differentiable (complex-valued) functions on \(\T\) endowed with the topology given by the semi-norms which define \(C^k(\T)\). The topological dual of this space, $\mathcal{E}'(\mathbb{T})$, is the set of distributions on $\mathbb{T}$. For any $f\in \mathcal{E}'(\mathbb{T})$, there exists a  Fourier series representation
\begin{equation*}
 f(x)=\sum_{m\in\mathbb{Z}}\hat{f}(j) \exp(\mathrm{i} m x),
\end{equation*}
where the (generalised) Fourier coeffecients \(\hat f(m)\) are of temperate growth in \(m\), that is, \(|\hat f(m)| \lesssim (1+ m^2)^{R/2}\) for some \(R>0\) and uniformly for all \(m \in \Z\). The Fourier coefficients may be calculated from the action of \(f\) on the functions \(\exp(-\mathrm{i} \cdot) \in \mathcal{E}(\T)\), via
\begin{equation*}
 \hat{f}(m) = f(\exp(-\mathrm{i}m \cdot)) = \int_{\T} f(x) \exp(-\mathrm{i}m x) \, \diff x,
\end{equation*}
where in the general case the integral must be understood only as a formal way of expressing the duality pairing on \(\mathcal{E}' \times \mathcal{E}\). One can then define the scale of Bessel-potential spaces
\begin{equation*}
 H^s(\mathbb{T}) =\Big\{f\in {\mathcal E}'(\mathbb{T}) \colon \sum_{m\in\mathbb{Z}}(1+m^2)^s|\hat{f}(m)|^2<\infty\Big\}, \qquad s \in \R,
\end{equation*}
equipped with the norm induced by the above expression. Note that we make use of negative indices \(s < 0\) in our existence proof, whence we may not define \(H^s(\T)\) just for \(s \geq 0\), in which case one could directly introduce it as a subspace of the standard Lebesgue space $L^2(\mathbb{T}) = H^0(\T)$. The spaces \(H^s(\T)\) has the same duality pairing and embedding properties as the spaces \(H^s(\R)\) (although even better). For more information on periodic function spaces and distributions, we refer the reader to \cite{triebel1983theory}.

The dispersive operator $\mathcal{K}$, too, can be transferred to the periodic setting. For any \(f\in L^{\infty}(\mathbb{T})\), one has
\begin{align*}
\mathcal{K}f(x) &= \int_\R K(x-y) f(y) \, \diff y\\
&=\int_{-\pi}^{\pi}\left(\sum_{k\in \mathbb{Z}}K(x-y+2\pi k)\right)f(y) \, \diff y\\
&= \int_{\T}K_{p}(x-y)f(y) \, \diff y,
\end{align*}
where
\[
K_p(x) =\sum_{m\in\mathbb{Z}}\frac{\tanh(m)}{m} \exp(\mathrm{i}mx).
\]
The function \(K_p\) furthermore belongs to \(L^2(\T)\), and the continuity of the map \(K_p \ast \colon H^s(\T) \to H^{s+1}(\T)\) follows from the decay of the Fourier coefficients of \(K_p\). For these and more facts on the periodic integral kernel \(K_p\), see  \cite{ehrnstrom2016existence}.

Finally, we only need to know that Lemma~\ref{lemma:mollification} holds true also in the periodic setting. For this, note that if $\varrho\in C^{\infty}_c(\mathbb{R})$ satisfies \(\int \varrho \, \diff x = 1\) with \(\varrho = 1\) in a neighbourhood of the origin, then the mollifier $\mathcal{J}_\varepsilon = \frac{1}{\varepsilon} \varrho(\cdot/\varepsilon)$ is given by 
\begin{equation}
  \widehat{\mathcal{J}_\varepsilon f}(m)=\hat{\varrho}(\varepsilon m)\hat{f}(m),
\end{equation}
analogous to the case on the line \cite{taylor1997partial}. Using the rapid decay of \(\hat \varrho\), one obtains the periodic equivalent of Lemma~\ref{lemma:mollification} (note here that we already know that \(\mathcal{J}_\varepsilon \in C(L^\infty(\T),{BUC}(\T))\), since that follows from \(\mathcal{J}_\varepsilon \in C(L^\infty(\R),BUC(\R))\)).

\begin{lemma*} For $k, l \in \Z_0$, one has
\begin{align*}
\|\mathcal{J}_\varepsilon f\|_{H^{k+l}(\mathbb{T})} &\lesssim \varepsilon^{-l}\|f\|_{H^k(\mathbb{T})},\\
\|(\mathcal{J}_\varepsilon-\mathcal{J}_{\varepsilon'})f\|_{H^{k}(\mathbb{T})} &\lesssim |\varepsilon-\varepsilon'| \|\partial_xf\|_{H^{k}(\mathbb{T})}.
\end{align*}
\end{lemma*}

\begin{proof} Since $\varrho\in C^{\infty}_c(\mathbb{R})\subset \mathcal{S}(\mathbb{R})$, one has $\hat{\varrho}\in \mathcal{S}(\mathbb{R})$ and for any given \(l \geq 0\), the estimate
\begin{equation*}
  |\hat{\varrho}(m)|\lesssim |m|^{-2l}
\end{equation*}
is uniform in \(m \in \Z \setminus \{0\}\). From this and $\widehat{\varrho}(0)= \int \varrho \, \diff x = 1$, we get
\begin{align*}
\|\mathcal{J}_\varepsilon f\|_{H^{k+l}(\mathbb{T})}^2
&=\sum_{m\in\mathbb{Z}}(1+m^2)^{k+l}|\hat{\varrho}(\varepsilon m)|^2|\hat{f}(m)|^2\\
&=\sum_{m\in\mathbb{Z}\setminus\{0\}}(1+m^2)^{k+l}|\hat{\varrho}(\varepsilon m)|^2|\hat{f}(m)|^2+|\hat{f}(0)|^2\\
&\lesssim \varepsilon^{-2l} \sum_{m\in\mathbb{Z}\setminus\{0\}}\frac{(1+m^2)^l}{m^{2l}}(1+m^2)^k|\hat{f}(m)|^2+|\hat{f}(0)|^2\\
&\lesssim \varepsilon^{-2l} \sum_{m\in\mathbb{Z}}(1+m^2)^k|\hat{f}(m)|^2\\
&= \varepsilon^{-2l} \|f\|_{H^k(\mathbb{R})}^2.
\end{align*}
To prove the second inequality, note that by the mean value theorem
\begin{align*}
  |\widehat{\varrho}(\varepsilon m)- \widehat{\varrho}(\varepsilon' m)|
  &= |\int_{\T} \varrho(x) (\exp(- \mathrm{i} \varepsilon m  x)-\exp(- \mathrm{i} \varepsilon' m  x))\, \diff x|\\
  &\lesssim |\varepsilon - \varepsilon'| m
\end{align*}
uniformly for all \(m \in \Z\). Therefore, 
\begin{align*}
\| (\mathcal{J}_\varepsilon - \mathcal{J}_{\varepsilon'}) f\|_{H^{k}(\mathbb{T})}^2
&=\sum_{m\in\mathbb{Z}}(\widehat{\varrho}(\varepsilon m)-\widehat{\varrho}(\varepsilon' m))^2(1+m^2)^{k}|\hat{f}(m)|^2\\
&=\sum_{m\neq 0}\frac{(\widehat{\varrho}(\varepsilon m)-\widehat{\varrho}(\varepsilon' m))^2}{m^2}(1+m^2)^k m^2 |\hat{f}(m)|^2\\
&\lesssim (\varepsilon - \varepsilon')^2 \| \partial_x f\|_{H^k(\mathbb{T})}^2.
\end{align*}
\end{proof}

It is similarly easy to see that the tame product estimate \eqref{eq:tame k} and the interpolation inequalities \eqref{10.6} and \eqref{13}  also hold on $\mathbb{T}$. Based on those inequalities and the above facts, one can follow the exact calculations carried out in Sections \ref{sec:prel}--\ref{proof of main theorem} to finish the proof of Theorem \ref{thm:main} on $\mathbb{T}$.


\end{document}